\documentclass[english,11pt]{smfart}
  \usepackage{smfthm,smfenum}
  \usepackage{amsmath,amsxtra,amscd,amssymb,latexsym,stmaryrd}
  \usepackage{mathrsfs,dsfont}
\usepackage[small,nohug,heads=vee]{diagrams} 
\usepackage{graphicx}
\usepackage{color}

\title[Geometric Wasserstein spaces]{A geometric study of Wasserstein spaces: Hadamard spaces}

\author{J\'er\^ome Bertrand}
\address{Institut de Math\'ematiques\\ Universit\'e Paul Sabatier \\ 118 route de
  Narbonne \\ F31062 Cedex 9 Toulouse\\ France}
\email{bertrand@math.univ-toulouse.fr}

\author{Beno\^{\i}t Kloeckner}
\address{Universit\'e de Grenoble I, Institut Fourier\\ CNRS UMR 5582\\ BP 74\\
  38402 Saint Martin d'H\`eres cedex\\ France}
\email{benoit.kloeckner@ujf-grenoble.fr}

\theoremstyle{plain}

\newcommand{\CAT}{\mathop {\mathrm{CAT}}\nolimits} 
\newcommand{\isom}{\mathop {\mathrm{Isom}}\nolimits}
\newcommand{\wass}{\mathop {\mathscr{W}_2}\nolimits}

\newcommand{\dw}{\mathop {\mathrm{W}}\nolimits}
\newcommand{\db}{d_{\infty}}
\newcommand{\dwb}{\mathrm{W}_{\infty}}

\newcommand{\supp}{\mathop {\mathrm{supp}}\nolimits}

\newcommand{\ep}{\varepsilon}
\newcommand{\id}{\mathrm{Id}}

\newcommand{\pr}{\mathscr{P}}

\renewcommand{\leq}{\leqslant}
\renewcommand{\geq}{\geqslant}

\begin{document}

\begin{abstract}
Optimal transport enables one to construct a metric on the set of
(sufficiently small at infinity) probability measures on any (not too wild) metric space $X$,
called its Wasserstein space $\wass(X)$.

In this paper we investigate the geometry of $\wass(X)$
when $X$ is a Hadamard space, by which we mean that $X$ has globally non-positive
sectional curvature and is locally compact.
Although it is known that --except in the case of the line-- $\wass(X)$
is not non-positively curved, our results show that $\wass(X)$ have large-scale properties
reminiscent of that of $X$. In particular we define a geodesic boundary for
$\wass(X)$ that enables us to prove a non-embeddablity result: if $X$ has the visibility
property, then the Euclidean plane does not admit any isometric embedding in $\wass(X)$.
\end{abstract}

\maketitle

\section{Introduction}

The goal of this paper is to contribute to the understanding of the 
geometry of Wasserstein spaces. Given a metric space $X$, the theory
of optimal transport (with quadratic cost)
gives birth to a new metric space, made of probability measures on $X$, 
often called its Wasserstein space and denoted here by $\wass(X)$
(precise definitions are recalled in the
first part of this paper). One can use this theory to study $X$, for example
by defining lower Ricci curvature bounds as in the celebrated works of
Lott-Villani \cite{Lott-Villani} and Sturm \cite{Sturm}. Conversely,
here we assume some
understanding of $X$ and try to use it to study geometric properties of $\wass(X)$.
A similar philosophy underlines the works of Lott in \cite{Lott} and
Takatsu and Yokota in \cite{Takatsu-Yokota}.

In a previous paper \cite{Kloeckner}, the second named author studied the case when $X$ is 
a Euclidean space.
Here we are interested in the far broader class of \emph{Hadamard spaces} which are
roughly the globally non-positively curved spaces. The first part of the paper gives 
the classical definitions and property we need both on optimal transport and Hadamard 
spaces; in particular the precise hypotheses
under which we shall work are given there (Definition \ref{defi:Hadamard}, see
also examples \ref{exem:Hadamard}).
Let us stress that we allow $X$ to be branching; trees, product
involving trees, some buildings are in particular treated in the same framework than, for example,
symmetric spaces of non-compact type.

While non-negative curvature is an assumption that is inherited by Wasserstein spaces,
it is well-known that non-positive curvature is not (an argument is recalled
in Section \ref{sec:interpolation}). We shall however show that some
features of Hadamard spaces still hold in their Wasserstein spaces. Let us now describe
the main results of the article.

A Hadamard space admits a well-known geometric compactification, obtained by adding a boundary 
at infinity made of asymptote classes of geodesic rays. In sections \ref{sec:displacement}
and \ref{sec:asymptotic}, we study the 
geodesic rays of $\wass(X)$. Using a displacement interpolation procedure (Proposition
\ref{prop:displacement}), we associate to each ray its \emph{asymptotic measure}
which lies in a subset $\mathscr{P}_1(c\partial X)$ of probability measures on 
the cone $c\partial X$ over the boundary of $X$ (Definition \ref{defi:asymptotic}).
It encodes the asymptotic distribution of the direction and speed of a measure 
running along the ray. Our
first main result is the \emph{asymptotic formula} (Theorem \ref{theo:asymptotic}) which enables
one to compute the asymptotic behavior of the distance between two rays in terms of 
the Wasserstein distance of the
 asymptotic measures, with respect to the angular cone distance on $c\partial X$. This asymptotic distance
is either bounded or asymptotically linear, so that the boundary $\partial\wass(X)$
of the Wasserstein space, defined as the set of asymptote classes of unit geodesic rays,
inherits an angular metric, just like $X$ does. A striking consequence of the asymptotic 
formula concerns the rank of $\wass(X)$, and partially
answers a question raised in the previous paper cited above.
\begin{theo}\label{theo:intro:rank}
If $X$ is a visibility space (e.g. if it has curvature bounded from above by a negative constant),
then it is not possible to embed the Euclidean plane isometrically in $\wass(X)$.
\end{theo}
In other words, when $X$ is strongly negatively curved --which implies that its has rank $1$--, then
although $\wass(X)$ is not negatively curved it has rank $1$ too. Note that our large-scale
method in fact implies more general non-embedding results, see Proposition \ref{prop:snowflaked}
and the discussion below. It is important to stress that asking
for \emph{isometric} embedding is the right regularity: more flexible conditions
are easily dealt with, see example \ref{exem:cylinder}. Besides this property on the rank, the 
Wasserstein space over a visibility space can also be differentiated from the one over an Euclidean 
space through its isometry group $ \isom \wass(X)$. Indeed, contrary to the Euclidean case 
where the isometry group $\isom \wass (\mathbb{R}^n)$ is larger than $\isom \mathbb{R}^n$ 
\cite{Kloeckner}, negatively curved spaces seem to have isometrically
rigid Wasserstein spaces in the sense that $\isom \wass (X)=\isom X$.
This holds at least in the case of manifolds and trees, as proved in a previous
version of this paper \cite{BK}; this result uses different methods from the ones
developped here and we aim at extending it, it shall therefore appear in a subsequent article.

Section \ref{part:boundary} is devoted to the definition of a so-called
\emph{cone topology} on $\partial\wass(X)$ and 
$\overline{\wass(X)}=\wass(X)\cup\partial\wass(X)$,
see Proposition \ref{prop:conetop}. Note that the angular metric
alluded to above, however useful and meaningful, does not define a satisfactory
topology (just as in $\partial X$, where the angular metric is usually not separable
and can even be discrete).
The point is that many monotony properties used
in the case of Hadamard spaces hold when one restricts to angles based at a Dirac mass. This
enables us to carry out the construction of this topology despite the presence of positive curvature.
The main result of this part is the following, restated as Theorem \ref{theo:homeobound}.
\begin{theo}
The asymptotic measure map defined from the boundary $\partial\wass(X)$
to the set of measures $\mathscr{P}_1(c\partial X)$ is a homeomorphism.
\end{theo}
Note that the two natural topologies on $\partial\wass(X)$,
namely the cone topology and the quotient topology of the topology of uniform
convergence on compact sets, coincide. The set $\mathscr{P}_1(c\partial X)$
is simply endowed with the weak topology (where the topology on $c\partial X$ is
induced by the cone topology of $\partial X$).

The possibility to identify $\partial\wass(X) $ to $\mathscr{P}_1(c\partial X)$
should be thought of as an interversion result, similar to displacement interpolation. The latter
says that ``a geodesic in the set of measures is a measure on the set of geodesics'', while the former
can be roughly restated as ``a boundary point of the set of measures is a measure on the (cone over the)
set of boundary points''. Note that $\overline{\wass(X)}$ is not compact; this is quite inevitable since
$\wass(X)$ is not locally compact.

\section{Reminders and notations}

As its title indicates, this part contains nothing new. We chose to give quite a lot
of recalls, so that the reader familiar with non-positively curved spaces can get a crash-course
on Wasserstein spaces, and the reader familiar with optimal transport can be introduced
to Hadamard spaces.

\subsection{Hadamard spaces}

Most properties of Hadamard spaces stated here are proved in \cite{Ballmann}.
Another more extensive reference is \cite{BH}.

\subsubsection{Geodesics}

Let us first fix some conventions for any metric space $Y$ (this letter shall be used
to design arbitrary spaces, while $X$ shall be reserved to the (Hadamard) space under study).

A \emph{geodesic} in $Y$
is a curve $\gamma:I\to Y$ defined on some interval $I$, such that there is
a constant $v$ that makes the following hold for all
times $t,t'$: 
\[d(\gamma_t,\gamma_{t'})=v|t-t'|.\]
In particular, all geodesics are assumed to be {\it globally} minimizing
and to have constant, non necessarily unitary speed. 
A metric space is \emph{geodesic} if any pair of points can be linked by a geodesic.

When $v=0$ we say that the geodesic is constant and it will be necessary
to consider this case. We denote by $\mathscr{G}^{T,T'}(Y)$ the set of geodesics
defined on the interval $[T,T']$. A \emph{geodesic ray} (or \emph{ray}, or \emph{complete ray})
is a geodesic defined on the interval $[0,+\infty)$.
A {\it complete geodesic} is a geodesic defined on $\mathbb{R}$. The set of 
rays is denoted by $\mathscr{R}(Y)$, 
the set of unit speed rays by $\mathscr{R}_1(Y)$ and the set of non-constant rays
by $\mathscr{R}_{>0}(Y)$. We shall also denote by $\mathscr{G}^{\mathbb{R}}(Y)$ the set 
of complete geodesics, and by $\mathscr{G}^{\mathbb{R}}_1(Y)$ the set of unit-speed complete
geodesics.

\subsubsection{Non-positive curvature}

A triangle in a geodesic space $Y$ is the datum of three points $(x,y,z)$ together 
with three geodesics parametrized on $[0,1]$
linking $x$ to $y$, $y$ to $z$ and $z$ to $x$. Given a triangle, one defines its
\emph{comparison triangle} $(\tilde x,\tilde y,\tilde z)$ as any triangle of the
Euclidean plane $\mathbb{R}^2$ that has the same side lengths: 
$d(x,y)=d(\tilde x,\tilde y)$, $d(x,z)=d(\tilde x,\tilde z)$
and  $d(y,z)=d(\tilde y,\tilde z)$. The comparison triangle is defined up to congruence.

A triangle with vertices $(x,y,z)$ is said to satisfy the $\CAT(0)$ inequality along
its $[yz]$ side (parametrized by a geodesic $\gamma \in \mathscr{G}^{0,1}(Y)$) if for all $t\in[0,1]$,
the following inequality holds:
\begin{equation}
d(x,\gamma(t))\leq d(\tilde x, (1-t)\tilde y+t\tilde z)
\label{eq:cat0}
\end{equation}
see figure \ref{fig:comparison}. A geodesic space is said to be locally $\CAT(0)$
if every point admits a neighborhood where all triangles satisfy the $\CAT(0)$ inequality
(along all there sides). When $Y$ is a Riemannian manifold, this is equivalent to ask
that $Y$ has non-positive sectional curvature.
A geodesic space is said to be globally $\CAT(0)$ if all its triangles
satisfy the $\CAT(0)$ inequality. Globally $\CAT(0)$ is equivalent to simply connected plus
locally $\CAT(0)$. We shall simply say $\CAT(0)$ for ``globally $\CAT(0)$'', but this is not
a universal convention.

\begin{figure}[ht]\begin{center}
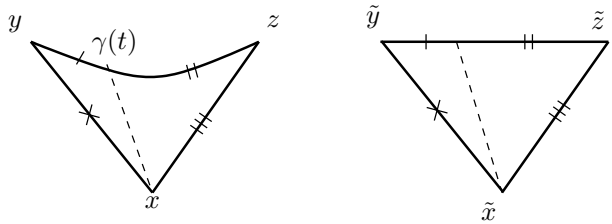
\caption{The $\CAT(0)$ inequality: the dashed segment is shorter in the original triangle
on the left than in the comparison triangle on the right.}\label{fig:comparison}
\end{center}\end{figure}

\subsubsection{Angles}

The $\CAT(0)$ condition can be translated in terms of angles as follows.
Given any geodesic triangle, choose any of its vertices, say $x$, and assume that the 
sides containing $x$ are parametrized by two geodesics $\sigma,\gamma \in \mathscr{G}^{0,1}(Y)$. 
If $Y$ is $\CAT(0)$, then the Euclidean angle $\widetilde{\angle}_{\widetilde{\gamma_s}\tilde x\widetilde{\sigma_t}}$
at $\tilde{x}$ is a nondecreasing function of $s$ and $t$.

One then defines in $Y$ the angle $\angle_{\gamma_1x\sigma_1}$ at $x$ 
as the limit, when $s$ and $t$ go to zero, of 
$\widetilde{\angle}_{\widetilde{\gamma_s}\tilde{x}\widetilde{\sigma_t}}$. As a consequence, one gets that 
for any geodesic triangle with vertices $(x,y,z)$ in a $\CAT(0)$ space, angle and comparison angle satisfy 
\[ \angle_{xyz} \leq \widetilde{\angle}_{\tilde{x}\tilde{y}\tilde{z}}. \]

\subsubsection{Distance convexity}

In a $\CAT(0)$ space, given two geodesics $\gamma$ 
and $\beta$, the distance function $t \longmapsto d(\gamma_t,\beta_t)$ is convex. 
This important property shall be kept in mind since it will be used very often in the sequel.

\subsubsection{Hadamard spaces: definition and examples}

We can now introduce the class of spaces we are interested in.
\begin{defi}\label{defi:Hadamard}
A metric space is a \emph{Hadamard space} if it is:
\begin{itemize}
\item Polish (i.e. complete and separable),
\item locally compact,
\item geodesic,
\item $\CAT(0)$, implying that it is simply connected.
\end{itemize}
\end{defi}
In all what follows, we consider a Hadamard space $X$. The Hadamard assumption 
may not always be made explicit, but the use of the letter $X$ shall always implicitely imply it.
Not all authors assume Hadamard spaces to be locally compact, and this assumption
excludes for example real trees.

\begin{exem}\label{exem:Hadamard}
There are many important examples of Hadamard spaces. Let us give some of them:
\begin{itemize}
\item the Euclidean space $\mathbb{R}^n$,
\item the real hyperbolic space $\mathbb{R}\mathrm{H}^n$,
\item the other hyperbolic spaces $\mathbb{C}\mathrm{H}^n$, $\mathbb{H}\mathrm{H}^n$, $\mathbb{O}\mathrm{H}^2$,
\item more generally the symmetric spaces of non-compact type, like the quotient
  $\mathrm{SL}(n;\mathbb{R})/\mathrm{SO}(n)$
  endowed with the metric induced by the Killing form of $\mathrm{SL}(n;\mathbb{R})$,
\item more generally any simply connected Riemannian manifold whose sectional curvature is 
  non-positive,
\item trees,
\item any product of Hadamard spaces,
\item some buildings, like product of trees having unit edges and no leaf or $I_{pq}$
      buildings (see \cite{Bourdon, Bourdon-Pajot}),
\item the gluing of any two Hadamard spaces along isometric, convex subsets; for example
   any Hadamard space with an additional geodesic ray glued at some point, or three hyperbolic 
   half-planes glued along their limiting geodesics, etc.
\end{itemize}
\end{exem}

\subsubsection{Geodesic boundary}

The construction of the geodesic boundary that we will shortly describe seems
to date back to \cite{Eberlein-ONeill}, but note that \cite{Busemann} is at the origin
of many related ideas.

Two rays of $X$ are {\it asymptotic} if they stay at bounded distance when
$t\to+\infty$, and this relation is denoted by $\sim$. The asymptote class
of a ray $\gamma$ is often denoted by $\gamma(\infty)$ or $\gamma_\infty$,
and is called the \emph{endpoint} or \emph{boundary point} of $\gamma$.

The \emph{geodesic} (or \emph{Hadamard}) \emph{boundary} of $X$
is defined as the set 
\[\partial X = \mathscr{R}_1(X)/\sim.\]
Using the convexity of distance along geodesics,
one can for example prove that, given points $x\in X$ and $\zeta\in\partial X$,
there is a unique unit ray starting at $x$ and ending at $\zeta$.

The union $\bar X=X\cup\partial X$ can be endowed with its so-called cone topology, which
makes $\bar X$ and $\partial X$ compact. 
Without entering into the details, let us say that
this topology induces the original topology on $X$, and that given a base point $x_0$
a basic neighborhood of a point $\zeta=\gamma(\infty)\in\partial X$ 
(where $\gamma$ starts at $x_0$) is the union, over all
rays $\sigma$ starting at $x_0$ such that $d(\sigma_t,\gamma_t)<\varepsilon$ for all $t<R$, of 
$\sigma([R,+\infty])$ (see figure \ref{fig:conetopology}).

\begin{figure}[htp]\begin{center}
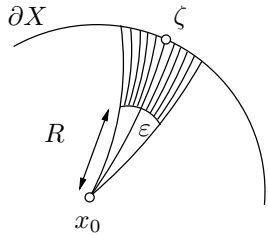
\caption{A basic neighborhood of a point $\zeta \in \partial X$ in the cone topology.}\label{fig:conetopology}
\end{center}\end{figure}

Consistently whith the cone topology, all previously defined sets of geodesics, as well as the larger sets 
$C(I,X)$ of continuous curves
defined on an interval $I$ with values in $X$, are endowed with the topology
of uniform convergence on compact sets. With this topology, since $X$ is Hadamard,
a geodesic segment is uniquely and continuously defined by its endpoints and a geodesic
ray is uniquely and continuously defined by its starting point, its endpoint in the boundary
and its speed. As a consequence, there are natural homeomorphisms
\[\mathscr{G}^{T,T'}(X)\simeq X^2,\qquad\mathscr{R}(X)\simeq X\times c\partial X\]
where $c\partial X$ is the cone over $\partial X$, that is the quotient of
$\partial X\times[0,+\infty)$ by the relation $(x,0)\sim (y,0)$ for all $x,y\in \partial X$.
We usually use the same notation $(x,s)$ for a couple and its equivalence class under this relation;
here $s$ shall be considered
as a speed. In particular one has $(\mathscr{R}(X)/\sim)\simeq c\partial X$.

Note that in view of our assumptions on $X$, all these spaces are locally compact
Polish topological spaces (that is, metrizable by a metric that is separable and complete).
It ensures that all finite measures on these spaces are Radon.

\subsubsection{Possible additional assumptions}

At some points, we shall make explicitely additional hypotheses on $X$. 
One says that a space is:
\begin{itemize}
\item \emph{geodesically complete} if every geodesic can be extended to a complete 
  geodesic,
\item \emph{non-branching} if two geodesics that coincide on an open set of parameters
   coincide everywhere on their common definition interval,
\item $\CAT(\kappa)$ fo some $\kappa<0$ if its triangles satisfy the \eqref{eq:cat0} inequality
  when the comparison triangle is taken in $\mathbb{R}\mathrm{H}^2_\kappa$, the hyperbolic plane of 
  curvature $\kappa$, instead of $\mathbb{R}^2$.
\item a \emph{visibility space} if for all pairs of distinct points $\alpha,\omega\in\partial X$ there is a complete
  geodesic $\gamma$ such that $\gamma(-\infty)=\alpha$ and $\gamma(+\infty)=\omega$.
\end{itemize}
Note that for all $\kappa<0$, the condition $\CAT(\kappa)$ implies that $X$ is a visibility space.

Geodesic completeness is quite mild (avoiding examples as trees with leafs), while the other are 
strong assumptions (for example non-branching rules out  trees and visibility rules out products).
Another possible assumption, that we shall not use directly, is for $X$ to have \emph{rank one},
meaning that it admits no isometric embedding of the Euclidean plane. It is a weaker condition than
visibility. More generally the
rank of a space $Y$ is the maximal dimension of an isometrically
embedded Eulidean space, and it has been proved a
very important invariant in the study of symmetric spaces. For example the hyperbolic spaces
$\mathbb{R}\mathrm{H}^n$, $\mathbb{C}\mathrm{H}^n$, $\mathbb{H}\mathrm{H}^n$ and 
$\mathbb{O}\mathrm{H}^2$ are the only rank one
symmetric spaces of non-compact type.

\subsubsection{Asymptotic distance}

Given two rays $\gamma,\sigma$, one defines their \emph{asymptotic distance} by
\[\db(\gamma,\beta)=\lim_{t\to+\infty} \frac{d(\gamma_t,\beta_t)}{t}.\]
This limit always exists because of the convexity of the distance function along
geodesics. Moreover $\db$ defines a metric on $c\partial X$ and,
by restriction, on $\partial X$ (in particular two rays whose distance grows sub-linearly must
be asymptotic). It can be proved that $d_\infty$ is the cone metric over $\partial X$ endowed with 
the angular metric. Namely, for any $(\xi,s), (\xi',t) \in c\partial X$, 
\begin{equation}\label{equa:conemetric}
 \db^2( (\xi,s), (\xi',t)) = s^2 +t^2 -2st \cos \angle (\xi,\xi')
 \end{equation}
 where $\angle (\xi,\xi')= \sup_{x \in X} \angle_x (\xi,\xi')$ is the supremum of angles between 
 rays issuing from $x$ and asympotic to $\xi$ and $\xi'$ respectively (we refer to \cite[Section 
 II.9]{BH} for more details and proof).

It is most important to keep in mind that the metric $\db$ \emph{does not induce the
cone topology}, but a much finer topology. 
The most extreme case is that of visibility spaces, where $\db(\gamma,\sigma)$ is $0$ if
$\gamma\sim\sigma$ and the sum of the speeds of $\gamma$ and $\sigma$ otherwise: the 
topology induced on $\partial X$ is discrete. However, the function $\db$ is 
lower semi-continuous with respect to the cone topology \cite[Proposition II.9.5]{BH},
so that it is a mesurable function.

In higher rank spaces, it can be useful to turn $\db$ into a length metric, called
the Tits metric, but we shall not use it so we refer the interested reader to the books
cited above. Let us just note that $\db$ resembles in some aspects the chordal metric
on a sphere. In particular, it is naturally isometric to this metric when $X$ is a Euclidean
space.

\subsection{The Wasserstein space}

In this  section, we recall the definition of Wasserstein space and some of its main properties.
For more details, we refer to the books \cite{Villani1} and \cite{Villani}.

\subsubsection{Optimal transport}

Let us start with the concept of optimal transport which is the theory aimed at studying the
\emph{Monge-Kantorovich} problem.

 Standard data for this problem are the following. We are given a Polish metric space $(Y,d)$, a 
 lower semicontinuous and nonnegative function
$c:Y\times Y\to \mathbb{R}^+$ called the cost function and two Borel probability measures $\mu,
\nu$ defined on $Y$.
A \emph{transport plan} between $\mu$ and $\nu$ is a measure on $Y\times Y$ whose marginals are
$\mu$ and $\nu$.
One should think of a transport plan as a specification of how the mass in $Y$, distributed 
according to $\mu$, is moved so as to be distributed according to $\nu$.
We denote by $\Gamma(\mu,\nu)$ the set of transport plans which is never empty (it contains $\mu\otimes\nu$)
and most of the time not reduced to one element. The Monge-Kantorovich problem is now
\[\min_{\Pi\in\Gamma(\mu,\nu)} \int_{Y\times Y} c(x,y) \,\Pi(dxdy)\]
where a minimizer is called an \emph{optimal transport plan}. The set of optimal transport plans is 
denoted by $\Gamma_o(\mu,\sigma)$.

Let us make a few comments on this problem. First, note that under these assumptions, the cost 
function is measurable (see, for instance, \cite[p. 26]{Villani1}). Now, existence of minimizers 
follows readily from the lower semicontinuity of the cost function together with the following 
compactness result which  will be used throughout this paper. We refer to \cite{Billingsley} for a 
proof.

\begin{theo}[Prokhorov's Theorem]\label{theo:prokhorov} Given a Polish space $Y$, a set $P \subset \pr(Y)$ 
is totally bounded (that is, has compact closure) for the weak topology if and only if it is tight,
namely for any $\ep >0$, there 
exists a compact set $K_{\ep}$ such that $\mu(Y \setminus  K_{\ep}) \leq \ep$ for any $\mu \in 
P$.
\end{theo}
For example, the set $\Gamma(\mu,\nu)$ is always compact.

We also mention that, compared to the existence problem, the issue of the \emph{uniqueness} of 
minimizers is 
considerably harder and requires, in general, additional assumptions. To conclude this 
introduction, we state a useful criterium to detect optimal transport plan among other plans, 
named cyclical monotonicity.

\begin{defi}[cyclical monotonocity]\label{defi:cyclical} Given a cost function 
$c: Y\times Y \longrightarrow \mathbb{R}^+$,  a set $\Gamma \subset Y \times Y$ is 
called $c$-cyclically monotone if for any finite family of pairs 
$(x_1,y_1),\cdots ,(x_m,y_m)$ in $\Gamma$, the following inequality holds
\begin{equation}\label{defi:cycli}
\sum_{i=1}^m c(x_i,y_{i+1}) \geq \sum_{i=1}^m c(x_i,y_i)
\end{equation}
where $y_{m+1}=y_1$. 
\end{defi}
In other words, a $c$-cyclically monotone set does not contain cycles of pairs (starting point, ending point)
along which a shift in the ending points would reduce the total cost. 

\begin{theo} Let $(Y,d)$ be a Polish space and $c: Y\times Y \longrightarrow \mathbb{R}^+$ 
be a lower semi-continuous cost function. Then, a transport plan is optimal relatively to 
$c$ if and only if it is concentrated
on a $c$-cyclically monotone set.

If $c$ is continuous, this is equivalent to its support being $c$-cyclically monotone.
\end{theo}
Under these assumptions, this result is due to Schachermayer and Teichmann 
\cite{Schachermayer}; see also \cite{Villani} for a proof.

\subsubsection{Wasserstein space}

Wasserstein spaces arise in a particular variant of the setting above.
\begin{defi}[Wasserstein space] Given a Polish metric space $Y$, its (quadratic) Wasserstein space $\wass(Y)$
is the set of Borel probability measures $\mu$ on $Y$ with  finite second moment, that is such that
\[\int_Y d(x_0,x)^2 \,\mu(dx)<+\infty \qquad\mbox{for some, hence all }x_0\in Y,\]
endowed with the Wasserstein metric defined by
\[ \dw^2(\mu_0,\mu_1) = \min_{\Pi \in \Gamma (\mu_0,\mu_1)} 
\int_{Y \times Y} d^2(x,y) \, \Pi(dx,dy).\]
\end{defi}
From now on, the cost $c$ will therefore be $c=d^2$.

The fact that $\dw$ is indeed a metric follows from the so-called ``gluing lemma'' which enables 
one
to propagate the triangular inequality, see e.g. \cite{Villani}. 
\begin{rema}
In this paper, we will also consider the Wasserstein space over the cone $c \partial X$ relative to the 
cost $\db^2$. Since $\db$ is lower semi-continuous, the usual theory of optimal transport
applies, and we do get a metric on the suitable space of measures. We 
shall denote by $\dwb$ the Wasserstein metric derived from $\db$.
\end{rema}
The Wasserstein space has several nice properties: it is Polish; 
it is compact as soon as $Y$ is, in which case the Wasserstein metric metrizes the weak topology; but
if $Y$ is not compact, then $\wass(Y)$ is not even locally compact and the Wasserstein metric induces
a topology stronger than the weak one (more precisely, convergence in Wasserstein distance
is equivalent to weak convergence plus convergence of the second moment). A very important
property is that $\wass(Y)$ is geodesic as soon as $Y$ is; let us give some details.

\subsubsection{Displacement interpolation}\label{sec:interpolation}

The proof of what we explain now can be found for example in chapter 7 of \cite{Villani},
see in particular corollary 7.22 and Theorem 7.30. 
Note that the concept of displacement interpolation has been introduced by McCann in 
\cite{McCann}. We write this section
in the case of a Hadamard space $X$, but most of it stays true for any Polish geodesic space.

\begin{defi}
Define a {\it dynamical transport plan} between two measures 
$\mu_0,\mu_1\in\wass(X)$ as a probability measure $\mu$ on $C([t_0,t_1];X)$ such that for
$i=0,1$ the law at time $t_i$ of a random curve drawn with law $\mu$ is $\mu_i$. In other words
we ask ${e_{t_i}}_\#\mu=\mu_i$ where $e_t$ is
the map $C([t_0,t_1];X)\to X$ defined by $e_t(\gamma)=\gamma_t$.

The {\it cost} of $\mu$ is then
\[|\mu|^2=\int \ell(\gamma)^2 \mu(d\gamma)\]
where $\ell(\gamma)$ is the length of the curve $\gamma$ (possibly $+\infty$).
A dynamical transport plan is {\it optimal} if it minimizes the cost over all dynamical
transport plans
\end{defi}

It is known that a dynamical transport
plan exists. Morevover, if $\mu$ is an optimal dynamical transport plan then:
\begin{enumerate}
\item the law $(e_{t_0},e_{t_1})_\#\mu$
of the couple $(\gamma_{t_0},\gamma_{t_1})$ where $\gamma$ is a random curve drawn with law $\mu$,
is an optimal transport plan between $\mu_0$ and $\mu_1$,
\item $\mu$-almost all $\gamma\in C([t_0,t_1];X)$ are geodesics.
\end{enumerate}

Conversely, if $\Pi$ is a (non-dynamical) optimal transport plan, then one can construct
for any $t_0<t_1$ an optimal dynamical transport plan by the following construction.
Let $F:X^2\to C([t_0,t_1];X)$ be the map that sends a couple $(x,y)$ of points  to
the unique geodesic parametrized on $[t_0,t_1]$ that starts at $x$ and ends at $y$.
Then $\mu=F_\#\Pi$ is an optimal dynamical transport plan, whose associated
optimal transport plan is obviously $\Pi$.

Given a dynamical plan $\mu \in \mathscr{P}(\mathscr{G}^I(X))$ where $I$ is an arbitrary interval, 
$\mu$ is said to be \emph{$c$-cyclically monotone} if for any $s,t \in I$, the support of the plan 
$(e_s,e_t)_{\#}\mu$ is $c$-cyclically monotone (note that here the cost is continuous).
As soon as $\mu_t$ have finite second moments, this is equivalent to $\mu$ being optimal,
but cyclical monotonicity has the advantage of being well-defined without any 
integrability assumption.

The main use of optimal dynamical transport plans is that they define geodesic segments.
Indeed, let as before $e_t$ be the map $\gamma\mapsto \gamma_t$ defined on the set of continuous
curves. If $\mu$ is an optimal dynamical transport plan, consider the law $\mu_t={e_t}_\#\mu$
at time $t$ of a random geodesic drawn with law $\mu$:
then $(\mu_t)_{t_0\leq t\leq t_1}$
is a geodesic of $\wass(X)$. Displacement interpolation is the converse to this principle.

\begin{prop}[Displacement interpolation]\label{prop:interpol}
Given $(\mu_t)_{t_0\leq t\leq t_1}$ any geodesic segment in $\wass(X)$, there is
a probability measure $\mu$ on $\mathscr{G}^{t_0,t_1}(X)$ such that for all $t$,
$\mu_t={e_t}_\#\mu$.
\end{prop}

If $\mu$ is a dynamical transport plan on $I$, for all $t_0,t_1\in I$
define the {\it time restriction} of $\mu$
to $[t_0,t_1]$ as $\mu^{t_0,t_1}={r_{t_0,t_1}}_\#\mu$ where $r_{t_0,t_1}(\gamma)$ is
the restriction of the curve $\gamma$ to the interval $[t_0,t_1]$.

Let $\mu$ be an optimal dynamical transport plan on $[0,1]$.
For all $t_0,t_1\in [0,1]$ the following holds (see \cite[Theorem 7.30]{Villani}):\begin{enumerate}
\item $\mu^{t_0,t_1}$ is an optimal dynamical transport plan,
\item if $X$ is non-branching and $(t_0,t_1)\neq(0,1)$, then $\mu^{t_0,t_1}$ is the {\it unique} 
(up to parametrization) optimal dynamical
transport plan between $\mu_{t_0}$ and $\mu_{t_1}$.
\end{enumerate}

\begin{rema}[The Wasserstein space is not non-positively curved]
It is well known that the non-positive curvature assumption on $X$ is not inherited by
$\wass(X)$ except if $X$ is the real line or a subset of it. Let us give a brief explanation
of this fact.

The affine structure makes the Wasserstein space contractible, so it cannot be locally
$\CAT(0)$ without being uniquely geodesic.
But as soon as $X$ is not reduced to a geodesic, there exists four
distinct points $x,x',y,z$ such that $d(x,y)=d(x,z)$ and $d(x',y)=d(x',z)$.
Between the measures $\mu=\frac12\delta_x+\frac12\delta_{x'}$ 
(where $\delta$'s are Dirac masses)
and $\nu=\frac12\delta_y+\frac12\delta_z$, all transport plans are optimal.
Each one of them defines a geodesic in $\wass(Y)$ from
$\mu$ to $\nu$, therefore $\wass(Y)$ is very far from being uniquely geodesic, and 
is in particular not non-positively curved.
\end{rema}

\section{Displacement interpolation for rays}\label{sec:displacement}

When $X$ is a Hadamard space,
the {\it geodesic boundary} of $\wass(X)$ is simply defined as
the set of asymptote classes of unitary geodesic rays:
\[\partial\wass(X)=\mathscr{R}_1(\wass(X))/\sim.\]

To study its structure, we need a good understanding of the geodesic rays in the Wasserstein
space, starting with the {\it displacement interpolation}
of rays. There is not much, but some work needed to extend the case of geodesic
segment; the crucial point being to handle the branching case.
Note that the case of complete geodesic is not different than that of the rays.

\begin{prop}[Displacement interpolation for rays in the non-branching case]\label{prop:displacement_nb}
If $X$ is non-branching, any geodesic ray $(\mu_t)$ of $\wass(X)$ admits a unique displacement interpolation,
that is a probability measure $\mu$ on $\mathscr{R}(X)$ such that $\mu_t$ is the law of the time
$t$ of a random ray having law $\mu$ (in other words, such that $\mu_t=(e_t)_\#\mu$).
\end{prop}

Note that since $(\mu_t)$ is a geodesic, the time restriction of $\mu$ to any segment
is an optimal dynamical transport plan.

\begin{proof}
Since $X$ is non-branching and $(\mu_t)$ is defined for all positive times, we know that there is
a unique optimal dynamical transport plan $\tilde\mu(T)$ from $\mu_0$ to $\mu_T$, and that
$\tilde\mu(T)=\tilde\mu(T')^{0,T}$ whenever $T<T'$ (see the discussion after Proposition 
\ref{prop:interpol}). The fact that $X$ is non-branching also
implies that for any two points $x,y\in X$ there is a unique maximal geodesic ray $F(x,y)$
starting at $x$ and passing at $y$ at time $1$; moreover this ray depends continuously
and therefore measurably on $(x,y)$. In other words, $(r_{0,1})_\#$ identifies 
$\mathscr{P}(\mathscr{G}^{0,T}(X))$ and $\mathscr{P}(\mathscr{G}^{0,1}(X))$.
 It follows that all $\tilde\mu(T)$ are uniquely defined
by $\tilde\mu(1)$ and that the probability measure $\mu=F_\#(\Pi)$, where
$\Pi$ is the optimal transport plan associated to $\tilde\mu(1)$, has the required property.
\end{proof}

When $X$ is branching, the previous proof fails for two reasons. The first one is that we cannot
determine $\tilde\mu(T)$ and define $\mu$ from $\tilde\mu(1)$ alone;
although Prokhorov's theorem will
do the trick. The second problem is that there may exist several optimal dynamical transport
plans corresponding to the same geodesic; but the set of these transports is always compact
and a diagonal process will solve the problem. However, we lose uniqueness in the process
and it would be interesting to single out one of the dynamical transport plan obtained.

\begin{prop}[Displacement interpolation for rays]\label{prop:displacement}
Any geodesic ray of $\wass(X)$ admits a displacement interpolation.
\end{prop}

\begin{proof}
Let $(\mu_t)_{t\geq 0}$ be a geodesic ray in $\wass(X)$ and
for all $T>0$, let $M(T)$ be the set of all optimal dynamical transport plans parametrized
on $[0,T]$ that induce the geodesic segment $(\mu_t)_{0\leq t\leq T}$.
It is a compact set according to \cite{Villani}, Corollary 7.22.
For all $T\in \mathbb{N}$, choose $\tilde\mu(T)$ in $M(T)$ and for all integer $0\leq T'\leq T$
define $\tilde\mu(T'|T)$ as the restriction $\tilde\mu(T)^{0,T'}$.

By a diagonal process, one can extract an increasing sequence $T_k$ of integers such that
for all $T'\in\mathbb{N}$, $\tilde\mu(T'|T_k)$ has a limit $\bar\mu(T')\in M(T')$ when $k\to+\infty$.
Since for all $T'<T$ the restriction maps
\[p^{T,T'} : \mathscr{G}^{0,T}\to \mathscr{G}^{0,T'}\]
are continuous, we get that
$p^{T,T'}(\bar\mu(T))=\bar\mu(T')$. We therefore have a projective system of measures; the
 projection map
\[p^T : \mathscr{R}(X)\to \mathscr{G}^{0,T}(X)\]
commutes with the $p^{T,T'}$ thus, according to a variant of Prokhorov's theorem \ref{theo:prokhorov} in the setting of projective system of measures \cite{Schwarz},
if we prove tightness, \emph{i.e.} that for all $\varepsilon>0$ there is a compact 
$K\subset\mathscr{R}(X)$ such that for all $T$,
$\bar\mu(T)(p^T K)\geq 1-\varepsilon$, then we can conclude
that there is a unique measure $\mu$ on $\mathscr{R}(X)$
such that $p^T_\#\mu=\bar\mu(T)$ for all $T$. This measure will have the required property
since $\bar\mu(T)\in M(T)$. 

Fix any $\varepsilon>0$. Let $K_0, K_1$ be compact subsets of $X$ such that
$\mu_i(K_i)\geq 1-\varepsilon/2$ for $i=0,1$. Let $K$ be the compact subset
of $\mathscr{R}(X)$
consisting in all geodesic rays starting in $K_0$ and whose time $1$ is in $K_1$.
Then for all
$T>1$, $\bar\mu(T)(p^T(K))\geq 1-\varepsilon$, as needed.
\end{proof}

The following result shall make displacement interpolation particularly useful.
\begin{lemm}[lifting]
Let $\mu$, $\sigma$ be probability measures on
$\mathscr{R}(X)$ (or similarly $\mathscr{G}^{T,T'}(X)$, ...) and denote
by $\mu_t=(e_t)_\#\mu$ and $\sigma_t=(e_t)_\#\sigma$ their time $t$.

Any transport plan $\Pi_t\in\Gamma(\mu_t,\sigma_t)$ admits a \emph{lift},
that is a transport plan $\Pi\in\Gamma(\mu,\sigma)$ such that
$\Pi_t=(e_t,e_t)_\#\Pi$.
\end{lemm}

Note that, as shall be apparent in the proof, the same holds for example
with the map $(e_s,e_t)$ when $s\neq t$.

\begin{proof}
Disintegrate $\mu$ along $\mu_t$: there is a family
$(\zeta_x)_{x\in X}$ of probability measures on $\mathscr{R}(X)$, each one
supported on the set $e_t^{-1}(x)$ of geodesic rays passing at $x$ at time $t$,
such that $\mu = \int \zeta_x \,\mu_t(dx)$ in the sense that
\[\mu(A)= \int_{\mathscr{R}(X)} \zeta_x(A)\mu_t(dx)\]
for all measurable $A$.
 Similarly, write $\sigma = \int \xi_y \,\sigma_t(dy)$ the disintegration
of $\sigma$ along $\sigma_t$.

Define then 
\[\Pi(A\times B) = \int_{\mathscr{R}(X)^2} \zeta_x(A)\xi_y(B)\Pi_t(dxdy).\]
It is a probability measure on $\mathscr{R}(X)^2$, and for any measurable 
sets $A,B$ in $X$ we have $\Pi(e_t^{-1}(A),e_t^{-1}(B))=\Pi_t(A\times B)$
because $\zeta_x\left(e_t^{-1}(A)\right)$ is $1$ if $x\in A$, $0$ otherwise
(and similarly for $\xi_y\left(e_t^{-1}(B)\right)$). A similar computation
gives that $\Pi$ has marginals $\mu$ and $\sigma$.
\end{proof}

Note that we gave the proof for the sake of completness, but one could simply apply twice
the gluing lemma, after noticing that the projection $e_t$ from $\mathscr{R}(X)$
to $X$ gives deterministic transport plans in $\Gamma(\mu,\mu_t)$ and
$\Gamma(\sigma,\sigma_t)$.

The lift of $\Pi_t$ need not be unique; the one constructed in the proof is very peculiar,
and can be called the \emph{most independent lift} of $\Pi_t$. It is well defined since
other disintegration families $(\zeta'_x)_x$ and $(\xi'_y)_y$ must coincide
with $(\zeta_x)_x$ and $(\xi_y)_y$ for $\mu_t$-almost all $x$ and $\sigma_t$-almost
all $y$ respectively.

The lifting lemma shall be used to translate the optimal transport problems
between $\mu_t$ and $\sigma_t$, where these measures move (usually along geodesics)
to transport problems between the fixed $\mu$ and $\sigma$, where it is the cost that moves.
In other words, we have just shown that minimizing $\int c(x,y)\, \Pi_t$ over the 
$\Pi_t\in\Gamma(\mu_t,\sigma_t)$
is the same than minimizing $\int c(\gamma_t,\beta_t)\, \Pi(d\gamma d\beta)$
over the $\Pi\in\Gamma(\mu,\sigma)$.

\section{Asymptotic measures}\label{sec:asymptotic}

Let us denote by $e_{\infty}$ the map defined by the formula
\begin{equation}
\begin{array}{rccc} 
e_{\infty} : & \mathscr{R}(X) &\longrightarrow & c\partial X \\
                  &         \gamma       & \longmapsto    & (  [\gamma^1],s(\gamma))
        \end{array}
\end{equation}       
where $\gamma^1$ is the unitary reparametrization of $\gamma$, $ [\gamma^1]$ is its
asymptote class and $s(\gamma)$ is the speed of $\gamma$.       
It is to be understood that whenever $s(\gamma)=0$, $[\gamma^1]$ can be taken arbitrarily in $\partial X$ and this choice does not matter.

\begin{defi}[asymptotic measure]\label{defi:asymptotic}
Let $(\mu_t)_{t \geq 0}$ be a geodesic ray in $\wass(X)$ and $\mu$ be a displacement interpolation
(so that $\mu_t= {e_t}_\#\mu$). We define the \emph{asymptotic measure} of the ray by
\[\mu_{\infty} := {e_{\infty}}_\#\mu.\]

We denote by $\mathscr{P}_1(c\partial X)$ the set of
probability measures $\nu$ on $c\partial X$ such that $\int v^2 \nu(dv)=1$.
\end{defi}

In the branching case, the dynamical optimal transport plan is not unique in general. 
Therefore, the asymptotic measure 
depends {\it a priori} on the choice of the dynamical optimal transport plan. We will see soon
that it is not the case.

Note that the speed of the geodesic $(\mu_t)$ is
\[\left(\int s^2(\gamma)\,\mu(d\gamma)\right)^{1/2}= \left(\int v^2 \,s_\#\mu(dv)\right)^{1/2}\]
and we denote it by $s(\mu)$. In particular, $\mathscr{P}_1(c\partial X)$ is the
set of measures that correspond to unit speed geodesics.
We shall use that the speed function $s$ defined in $\mathscr{R}(X)$
is in $L^2(\mu)$ several times.

The main result of this section is the following.
\begin{theo}[asymptotic formula]\label{theo:asymptotic}
Consider two geodesic rays $(\mu_t)_{t \geq 0}$ and $(\sigma_t)_{t \geq 0}$, let
$\mu$ and $\sigma$ be any of their displacement interpolations and $\mu_\infty$, $\sigma_\infty$
be the corresponding asymptotic measures.
Then $(\mu_t)$ and $(\sigma_t)$ are asymptotic if and only if $\mu_\infty=\sigma_\infty$, and we
have
\[\lim_{t\to\infty} \frac{\dw(\mu_t,\sigma_t) }{t}= \dwb (\mu_\infty,\sigma_\infty).\]
\end{theo}

Therefore: as in $X$ itself, the distance between two rays is either bounded or of linear growth,
and two displacement interpolations of the same ray define the same asymptotic measure.

The r\^ole of the asymptotic formula goes far beyond justifying Definition \ref{defi:asymptotic} 
in the branching case: it
gives us a very good control on geodesic rays of $\wass(X)$ on which several of our results 
rely. To cite one, the asymptotic formula is the main ingredient of Theorem \ref{theo:intro:rank} on the rank of $\wass(X)$.

For every $t\geq 0$, let $d_t$ be the function defined on $\mathscr{R}(X)\times\mathscr{R}(X)$
by $d_t(\gamma,\beta)=d(\gamma_t,\beta_t)$. We start with an implementation of a classical principle.

\begin{lemm}
The function $d_t$ is in $L^2(\Gamma(\mu,\sigma))$,
by which we mean that there is a constant $C=C(\mu,\sigma)$ such that for all $\Pi\in\Gamma(\mu,\sigma)$,
$\int d_t^2 \,\Pi\leq C$.
\end{lemm}
In the following, it will be of primary importance that $C$ does not depend on $\Pi$.

\begin{proof}
Denoting by $x$ any base point in $X$ we have
\begin{eqnarray*}
\int d_t^2\, \Pi &\leq& \int \left(d(\gamma_t,x)+d(x,\beta_t)\right)^2 \,\Pi(d\gamma d\beta)\\
  &\leq& 2\int \left( d^2(\gamma_t,x) + d^2(x,\beta_t)\right) \,\Pi(d\gamma d\beta)\\
  &=& 2\int d^2(\gamma_t,x)\, \mu(d\gamma) + 2\int d^2(x,\beta_t) \,\sigma(d\beta)\\
  &=& 2\dw^2(\mu_t,\delta_x)+2\dw^2(\sigma_t,\delta_x).
\end{eqnarray*}
\end{proof}

\begin{proof}[Proof of the asymptotic formula]
We shall use several times the following translation of the
convexity of the distance function: given any geodesic rays $\gamma,\beta$ of $X$,
the function
\[ f_t(\gamma,\beta):=\frac{d(\gamma_t,\beta_t)-d(\gamma_0,\beta_0)}{t}\]
is nondecreasing in $t$ and has limit $\db(\gamma,\beta)$.

Assume first that $\mu_\infty=\sigma_\infty$, and let us prove that $\dw(\mu_t,\sigma_t)$ is bounded.
The lifting lemma gives us a transport plan $\Pi\in\Gamma(\mu,\sigma)$ such that for all
$(\gamma,\beta)$ in its support, $\gamma_\infty=\beta_\infty$ (simply observe that the lifting
Lemma applies equally well to $t=\infty$, and lift the trivial transport
$(\id\times\id)_\#\mu_\infty$). If $f_t(\gamma,\beta)$ is positive for some $t$, then $\db(\gamma,\beta)>0$.
It follows that on $\supp\Pi$, $d_t\leq d_0$. Therefore:
\[\dw(\mu_t,\sigma_t)\leq \int d_t^2 \, \Pi \leq \int d_0^2 \, \Pi\]
which is bounded by the previous lemma.

Let now $\Pi_\infty$ be a transport plan from $\mu_\infty$ to $\sigma_\infty$ that is optimal with respect
to $d_\infty$. Such a minimizer exists 
since $d_\infty$ is non-negative and lower-semicontinuous with respect to the cone topology of 
$c\partial X$; note that taking an almost minimizer would be sufficient 
anyway. Denote by $\tilde\Pi$ a lift of $\Pi_\infty$ to $\Gamma(\mu,\sigma)$; then
\[\frac{\dw^2(\mu_t,\sigma_t)}{t^2}\leq \int \frac{d_t^2}{t^2}\, \tilde\Pi. \]
We have $2\geq d_\infty\geq f_t\geq f_1\geq -d_0$ for all $t\geq1$, so that
$t^{-1}d_t$ is bounded by $2+d_0$ and $- d_0$. We can thus apply the dominated convergence theorem,
which gives 
\[\limsup \frac{\dw^2(\mu_t,\sigma_t)}{t^2} \leq \int \db^2\, \tilde\Pi = \dwb^2(\mu_\infty,\sigma_\infty).\]

To prove the other inequality, we introduce $g_t:=\max(0,f_t)$. It is a
nondecreasing, nonnegative function with $\db$ as limit, and it satisfies
$t^2 g_t^2 \leq d_t^2$.

Let $\Pi_t$ be an optimal transport plan between $(\mu_t)$ and $(\sigma_t)$ and 
$\tilde\Pi_t$ be a lift to $\Gamma(\mu,\sigma)$, which by Prokhorov Theorem is compact in the weak topology.
Let $(t_k)_k$ be an increasing sequence such that 
\[\lim_k \dw(\mu_{t_k},\sigma_{t_k})=\liminf_t \dw(\mu_t,\sigma_t)\]
and $\tilde\Pi_{t_k}$ weakly converges to some $\tilde\Pi_\infty$.

For all $k'<k$, we have
\begin{eqnarray*}
\frac{\dw^2(\mu_{t_k},\sigma_{t_k})}{t_k^2} &=& \frac1{t_k^2}\int d_{t_k}^2\,\tilde\Pi_{t_k} \\
                                  &\geq& \int g_{t_k}^2\, \tilde\Pi_{t_k} \\
                                  &\geq& \int g_{t_{k'}}^2\, \tilde\Pi_{t_k}.
\end{eqnarray*}
Letting $k\to\infty$ and using that the $g_t$ are continuous, we obtain
\[\liminf_t \frac{\dw^2(\mu_t,\sigma_t)}{t^2} \geq \int g_{t_{k'}}^2\, \tilde\Pi_\infty\]
for all $k'$. But $g_{t_{k'}}\leq 2$ and the dominated convergence
theorem enables us to let $k'\to\infty$:
\[\liminf_t \frac{\dw^2(\mu_t,\sigma_t)}{t^2} \geq \int \db^2\, \tilde\Pi_\infty \geq
  \dw^2(\mu_\infty,\sigma_\infty).\]
This ends the proof of the asymptotic formula, and shows that if $(\mu_t)$ and $(\sigma_t)$ stay
at bounded distance then $\mu_\infty=\sigma_\infty$.
\end{proof}

\section{Complete geodesics and the rank}\label{sec:rank}

In this section we study complete geodesics in $\wass(X)$, in particular to
understand its rank. Recall that the \emph{rank} of a metric space
is the highest dimension of a Euclidean space that embeds
isometrically in it.

The main result of this section is the following.
\begin{theo}\label{theo:rank}
If $X$ is a visibility space, in particular if it is $\CAT(\kappa)$
with $\kappa<0$, then $\wass(X)$ has rank $1$.
\end{theo}

We expect more generally that for most, if not all Hadamard space $X$, the rank
of $\wass(X)$ is equal to the rank of $X$. Theorem \ref{theo:rank} 
is a first step in this direction.
Note that the fact that we ask the embedding to be nothing weaker than an isometry
and to be global is
important, as shown by example \ref{exem:cylinder} at the end of the section.

To prove Theorem \ref{theo:rank}, let us first see that the asymptotic measure 
of a complete geodesic is much more constrained than that of a mere ray.
\begin{prop}\label{prop:completegeodesics}
Let $(\mu_t)_{t\in\mathbb{R}}$ be a complete unit speed geodesic of $\wass(X)$, and $\mu$
be one of its displacement interpolations. Then $\mu$ is concentrated 
on the set of \emph{unit speed} geodesics of $X$.
\end{prop}
Note that here we do not use any assumption on $X$ (besides the existence of displacement
interpolations).

\begin{proof}
Let $\gamma,\beta$ be two geodesics in the support of $\mu$ and let $a=s(\gamma)$ and $b=s(\beta)$.
Fix some point
$x\in X$. Then we have the equivalents $d(x,\gamma_t)\sim at$ and
$d(x,\beta_t)\sim bt$ when $t\to\pm\infty$. In particular, we get
that 
\begin{eqnarray*}
d^2(\gamma_t,\beta_{-t}) &\leq& \left(d(\gamma_t,x)+d(x,\beta_{-t})\right)^2\\
                       &\leq& (a+b)^2t^2 + o(t^2)
\end{eqnarray*}
and similarly $d^2(\beta_t,\gamma_{-t})\leq (a+b)^2t^2 + o(t^2)$.
We also have $d^2(\gamma_t,\gamma_{-t})=4a^2t^2$ and $d^2(\beta_t,\beta_{-t})=4b^2t^2$.
But since $(\mu_t)$ is a geodesic, the transport plan $\Pi^t \in \Gamma (\mu_t,\mu_{-t})$ induced by  $\mu$ must
respect the cyclical monotonicity. In particular we have
\[d^2(\gamma_t,\gamma_{-t})+d^2(\beta_t,\beta_{-t})\leq d^2(\gamma_t,\beta_{-t})+d^2(\beta_t,\gamma_{-t}).\]
From this and letting $t\to\infty$ we get 
$4a^2 +4b^2\leq 2a^2+2b^2+4ab$, which is only possible when $a=b$.

We proved that the speed of geodesics in the support of $\mu$ is constant, and
since their square integrates (with respect to $\mu$) to $1$ we get
the desired conclusion.
\end{proof}

As a consequence, the asymptotic measure of a unit speed ray that can be extended to a complete
geodesic lies in the subset $\mathscr{P}(\partial X)$ of $\mathscr{P}_1(c\partial X)$
(where we identify a space $Y$ with the level $Y\times \{1\}$ in its cone).
\begin{prop}\label{prop:snowflaked}
If $X$ is a visibility space, the space $\mathscr{P}(\partial X)$
endowed with the metric $\dwb$ (where a ray is identified with its asymptote class)
contains no non-constant rectifiable curve.
\end{prop}

\begin{proof}
First, the visibility assumption implies
that whenever $\gamma$ and $\beta$ are non asymptotic,
unit speed geodesic rays of $X$, we have $d_\infty(\gamma,\beta)=2$. Let us prove that for all displacement interpolation
$\mu,\sigma$ of rays in $\wass(X)$, assumed to be concentrated on $\mathscr{R}_1(X)$, we have
\begin{equation}
\dwb(\mu_\infty,\sigma_\infty)=2|\mu_\infty-\sigma_\infty|_v^{1/2} \label{eq:negcurvdistance}
\end{equation}
where $|\cdot|_v$ is the total variation norm. Let us recall that, using the Jordan measure
decomposition of $\mu_\infty-\sigma_\infty$, one can find positive (not probability)
measures $\mu'_\infty$, $\sigma'_\infty$
and $\nu$ such that $\mu'_\infty$ and $\sigma'_\infty$ are mutually singular,
$\mu_\infty=\nu+\mu'_\infty$ and $\sigma_\infty=\nu+\sigma'_\infty$. By definition,
$|\mu_\infty-\sigma_\infty|_v$ is the total mass of $\mu'_\infty$ (or, equivalently, 
of $\sigma'_\infty$). To find an optimal transport plan, one can simply leave the common
mass in place and move arbitrarily what is left, for example taking
$$\Pi=(\mathrm{Id},\mathrm{Id})_\#\nu+\mu'_\infty\otimes\sigma'_\infty$$
where $\otimes$ denotes the product measure \emph{normalized to have the same mass as
each factor}.
Since $\Pi$ moves a mass $|\mu_\infty-\sigma_\infty|_v$ by a distance of $2$,
it has cost $4|\mu_\infty-\sigma_\infty|_v$. More generally, due to the behavior
of $d_\infty$, any transport plan $\Pi'$ has cost 
\[4\Pi'(\{(\zeta,\xi) \,|\, \zeta\neq \xi \})\]
which cannot be smaller than the cost of $\Pi$.

The proposition now results from the more general following lemma, which is 
well-known at least in the case of Euclidean space.
\begin{lemm}[Snowflaked metrics]
Let $(Y,d)$ be any metric space and $\alpha<1$ be a positive number. Then
$(Y,d^\alpha)$ is a metric space not containing any non-constant rectifiable curve.
\end{lemm}

The fact that $d^\alpha$ is a metric comes from the inequality
$a+b\leq(a^\alpha+b^\alpha)^{1/\alpha}$ for positive $a,b$.

Let $c:I\to Y$ be a non-constant curve. Up to restriction and reparametrization, 
we can assume that $I=[0,1]$ and $c_0\neq c_1$. Take any positive integer $n$;
since $d$ is continuous, by the intermediate value theorem there are
numbers $t_1=0<t_2<\cdots<t_n<1$ such that $d(c_{t_{i-1}},c_{t_i})= d(c_0,c_1)/n$
and $d(c_{t_n},c_1)\geq d(c_0,c_1)/n$.
Denoting by $\ell$ the length according to the ``snowflaked'' metric
$d^\alpha$, we get that
\[\ell(c)\geq n\left(\frac{d(c_0,c_1)}{n}\right)^\alpha
  \geq d^\alpha(c_0,c_1) n^{1-\alpha}.\]
Since this holds for all $n$, $\ell(c)=\infty$ and $c$ is not rectifiable.
\end{proof}

\begin{proof}[Proof of Theorem \ref{theo:rank}]
Assume that there is an isometric embedding $\varphi:\mathbb{R}^2\to\wass(X)$.
Let $r^\theta$ be the ray starting at the origin and making an angle $\theta$
with some fixed direction. Then, since $r^\theta$ extends to a complete
geodesic, so does $\varphi\circ r^\theta$. The displacement interpolation
 $\mu^\theta$ of this ray of $\wass(X)$ must be concentrated on
$\mathscr{R}_1(X)$ by Proposition \ref{prop:completegeodesics}, so 
that $\mu^\theta_\infty\in\mathscr{P}(\partial X)$. But $\varphi$ being isometric,
the map $\theta\to\mu^\theta_\infty$ should be an isometric embedding
from the boundary of $\mathbb{R}^2$ (that is, the unit circle endowed with the 
chordal metric) to $(\mathscr{P}(\partial X),\dwb)$.
In particular its image would be a non-constant rectifiable curve, in contradiction
with Proposition \ref{prop:snowflaked}.
\end{proof}

Note that the same method yields more general results: we can rule out
the isometric embedding of Minkowski planes ($\mathbb{R}^2$ endowed with any
norm), and of their cones
of the form $\{x^2<\varepsilon y^2\}$ for any $\varepsilon>0$.
This contrasts with above-mentionned fact that even when $X$ is reduced to a line,
some Euclidean \emph{half}-cones of arbitrary dimension embeds isometrically
in $\wass(X)$.

\begin{exem}\label{exem:cylinder}
Let us remind an example of \cite{Kloeckner} showing that there are plenty
 of weaker Euclidean embedding in most Wasserstein spaces. 
Consider the set $\mathbb{R}^n_<$ of increasingly ordered real $n$-tuples.
The map
\begin{eqnarray*}
f: \mathbb{R}^n_<  &\to& \wass(\mathbb{R}) \\
   (x_1,\dots,x_n) &\mapsto& \frac1n \delta_{\sqrt{n} x_1}+\dots+\frac1n \delta_{\sqrt{n} x_n}
\end{eqnarray*}
is an isometric embedding. Assume $X$ contains a complete geodesic; then its Wasserstein
space contains a copy of $\wass(\mathbb{R})$ so that  $\mathbb{R}^n_<$ embeds
isometrically into $\wass(X)$. But $\mathbb{R}^n_<$ is an open half-cone invariant
under a $1$-parameter group of translations so that:
\begin{itemize}
\item $\wass(X)$ contains round Euclidean half-cones of arbitrary dimension,
      in particular $\mathbb{R}^{n-1}$ admits bi-Lipschitz emdedding in  $\wass(X)$
      for arbitrary $n$,
\item the cylinder $\mathbb{R}\times B^n(r)$ where $B^n(r)$ is a Euclidean ball of 
      arbitrary radius $r$ embeds isometrically in $\wass(X)$ for all $n$.
\end{itemize}
We see that under a very mild assumption, most weak ranks of $\wass(X)$ are infinite
(examples of weak ranks include the largest dimension of a bi-Lipshitz embedded
Euclidean space and the largest rank of a quasi-isometrically embedded $\mathbb{Z}^n$).
\end{exem}

Another customary definition of rank for Hadamard manifolds uses Jacobi fields;
one could try to extend the
notion of Jacobi fields in the setting of Wasserstein spaces, but our example
seems to indicate that with such a definition even $\wass(\mathbb{R})$ should have infinite 
rank. Let us give a precise result showing this.

\begin{prop}
Given any geodesic $(\mu_t)$ in $\wass(\mathbb{R})$ and any $n\in\mathbb{N}$,
there are independent unit vectors $v_1,\dots,v_n$ in $\mathbb{R}^n$ and a positive
$\varepsilon$ such that the set
\[C=\left\{t v_1 + s_2 v_2 +\dots + s_n v_n \,|\, t\in\mathbb{R}, 
           s_i\in [0,\varepsilon) \right\}\]
embeds isometrically in $\wass(\mathbb{R})$, with $tv_1$ sent to $\mu_t$ for all $t$.
\end{prop}

In particular, for any sensible metric definition of Jacobi fields, a geodesic admits
arbitrarily many independent pairwise commuting Jacobi fields.

\begin{proof}
It is easy to see and proved in \cite{Kloeckner} that there is some measure
$\mu$ on $\mathbb{R}$ such that $\mu_t=T_{t\#}\mu$ for all $t$, where $T_t$
is the translation $x\mapsto x+t$.

If $\mu$ has finite support, then an embedding similar to the one described
in example \ref{exem:cylinder} gives the conclusion. Assume that the support of
$\mu$ contains at least $n-1$ points $x_2,\dots,x_n$ and choose disjoint neighborhoods
$U_2,\dots,U_n$ of these points.
Let $X_2,\dots,X_n$ be smooth vector fields on $\mathbb{R}$,
each $X_i$ having support in $U_i$, define
$\Phi(t,s_2,\dots,s_n)(x)=x+t+\sum_i s_i X_i(x)$ and consider the map 
\begin{eqnarray*}
f : \mathbb{R}\times [0,+\infty)^{n-1} &\to& \wass(\mathbb{R}) \\
    (t,s_2,\dots,s_n) &\mapsto& \Phi(t,s_2,\dots,s_n)_\# \mu
\end{eqnarray*}
Using convexity of the cost and cyclical monotonocity, one sees that optimal transport
plans in $\wass(\mathbb{R})$ are exactly those where no inversion of mass occurs
(i.e. whose support does not contain pairs $(x,y)$ and $(x',y')$ such that
$x>x'$ and $y<y'$). 
In particular, the push forward by the map
$\Phi(t,s_2,\dots,s_n)$
defines an optimal transport plan between a measure and its image
as soon as $s_i$'s are small enough. It follows that for small enough $s_i$'s
we have 
\[\begin{split}
\dw\big(f(t,(s_i)),f(t',(s_i'))\big)^2 = & (t-t')^2 
  + \sum_i(s_i-s_i')^2\left(\int_{U_i}X_i^2(x)\,\mu(dx)\right)  \\
  &+2(t-t')\sum_i (s_i-s'_i) \int_{U_i} X_i(x) \,\mu(dx)
\end{split}\]
which is a quadratic expression in $(t-t',s_2-s_2',\dots,s_n-s_n')$.
Modifying $f$ by a linear change of coordinates, we get the desired embedding.
\end{proof}

Note that when the support of $\mu$ contains at least $n$ points,
we can in fact construct a $n$ dimensional uniformly large Euclidean neighborhood of $(\mu_t)$.
However, we only get the announced ``corner'' when the support of $\mu$ is too small,
in particular when $\mu$ is a Dirac mass.

As a last remark, let us point out that the Jacobi fields constructed here have well-defined
flows only for small times. The study (and definition) of fully integrable Jacobi
fields on $\wass(X)$ could lead to an understanding of its flats even when $X$ has higher rank.

\section{The geodesic boundary and its cone topology}\label{part:boundary}

In this section, we adapt to $\wass (X)$ the classical construction of the 
cone topology on the geodesic compactification of Hadamard space, 
see for instance \cite{Ballmann}. We introduce  this 
topology on $\overline{\wass(X)}=\wass(X)\cup\partial\wass(X)$. We shall prove in 
Proposition \ref{prop:conetop} that the cone topology turns $\overline{\wass(X)}$ 
into a first-countable Hausdorff space and that the topology induced on 
$\wass(X)$ coincides with the topology derived from the Wasserstein metric.

In the next section, we shall show  in Theorem \ref{theo:homeobound} that $\partial\wass(X)$
is homeomorphic to 
$$\pr_1(c\partial X)= \left\{\zeta \in \pr(c\partial X) ; \int s^2 \,\zeta(d\xi,ds)=1\right\}$$
endowed with the weak topology. In Corollary \ref{coro:homeo}, we rewrite the above result in 
terms of Wasserstein space over $c \partial X$.

From now on, we will use the following notations.
Let $(Y,d)$ be a geodesic space and $y \in Y$. We set $\mathscr{R}_{y}(Y)$ (respectively 
$\mathscr{R}_{y,1}(Y)$) the set of geodesic rays in $Y$ starting at $y$ 
(respectively  the set of unitary geodesic rays starting at $y$). These sets are closed subsets of 
$\mathscr{R}(Y)$ endowed with the topology of uniform 
convergence on compact subsets.

 The cone topology 
 on $\overline{\wass(X)}=\wass(X)\cup\partial\wass(X)$
 is defined by using as a basis the open sets of $\wass(X)$ together with
$$ U(x, \xi,R , \ep)= \big\{ \theta \in \overline{\wass(X)} ; \theta \not\in \overline{B}(\delta_x, R), 
\dw( (\mu_{\delta_x,\theta})_R , (\mu_{\delta_x,
\xi})_R) < \ep \big\}$$
where $x \in X$ is a fixed point, $\xi$ runs over $\partial\wass(X)$, $R$ and $\varepsilon$ run
over $(0,+\infty)$ and $\mu_{\delta_x,\theta}$ is the unitary geodesic between $\delta_x$
 and $\theta$ 
(existence and uniqueness follow from Lemma
\ref{lemm:existray}).

The main properties of the cone topology on $\overline{\wass(X)}$ are gathered together in the following proposition.
\begin{prop}\label{prop:conetop} The cone topology on  $\overline{\wass(X)}$ is well-defined and is independent of the choice of the basepoint $\delta_x$. Moreover, endowed with this topology,  $\overline{\wass(X)}$ is a first-countable Hausdorff space. By definition, the topology induced on
$\wass(X)$ coincides with the topology derived from the Wasserstein metric.
\end{prop}

\begin{rema}\label{rema:firstc} We emphasize that the topology induced 
 on $\partial \wass(X)$ by the cone topology coincides 
 with the quotient topology induced by the topology of uniform convergence on compact 
 subsets on 
 the set of unitary rays in $\wass(X)$.
 
Moreover, since $\partial \wass(X)$ endowed with the cone topology is first-countable, 
continuity and sequential continuity are equivalent in this topological space. 
 \end{rema}
 The scheme of proof is the same as in the 
nonpositively curved case. However, to get the result, we first need to generalize to our setting 
some 
properties related to nonpositive curvature.

\begin{lemm}\label{lemm:existray} Given $x \in X$,  the set of unitary rays in  $\wass(X)$ starting at $\delta_x$ is in one-to-one correspondence with the set
$\pr_1(c\partial X)$. Moreover, for any $\xi \in \partial \wass(X)$, there 
exists a unique unitary ray starting at $\delta_x$ and belonging to $\xi$.
\end{lemm}
\begin{proof} Recall that there exists a unique transport plan between a Dirac mass and any 
measure in $\wass(X)$. Since there is a unique geodesic between two given points in $X$, the 
same property remains true for dynamical transportation plans. Using the previous remarks, we get that any $\mu \in \pr (\mathscr{R}_{x}(X))$ such that
$ \int s^2(\gamma) \mu(d\gamma)< + \infty$ induces a ray starting at $\delta_x$.  Moreover, since displacement interpolation always exists (see Proposition \ref{prop:displacement}),  the set $\mathscr{R}_{\delta_x,1}(\wass(X))$ is in one-to-one correspondence with the unitary dynamical transportation plans starting at $\delta_x$, namely with the measures $\mu \in  \pr (\mathscr{R}_{x}(X))$ such that $ \int s^2(\gamma) \mu(d\gamma)=1$. Now, since $X$ is a 
Hadamard space, we recall that $\mathscr{R}(X)$ is homeomorphic to $ X \times c\partial X$ (where the left coordinate is the initial location of the ray). Therefore, for any $x \in X$, the 
previous map induces a homeomorphism
\begin{equation}\label{equa:phix}
 \phi_x : \mathscr{R}_{x}(X) \longrightarrow c \partial X.
 \end{equation}
This gives us a one-to-one 
correspondence between the set of unitary dynamical transportation plans starting at $\delta_x$ and the set $\pr_1(c\partial X)$. 

Given $\xi  \in \partial \wass(X)$, consider a unit ray $(\mu_t)$ in
$\wass(X)$ belonging to $\xi$ and $\mu_{\infty} \in \pr_1(c\partial X)$ its asymptotic measure.  We claim that ${\phi_x^{-1}}_{\#}\mu_{\infty}$ is the unique ray starting at $\delta_x$ and belonging to $\xi$. Indeed,  ${\phi_x^{-1}}_{\#}\mu_{\infty}$ and 
$(\mu_t)$ have the same asymptotic measure, thus they are asymptotic thanks to the asymptotic 
formula (Theorem \ref{theo:asymptotic}). The asymptotic formula also implies that two asymptotic rays (starting at $\delta_x$) have the same asymptotic measure, thus they are equal thanks to the first part of the lemma.
\end{proof}
\begin{lemm}\label{lemm:triangle} Let $(\mu_t),(\sigma_t)$ be two unitary geodesics (possibly rays) in
 $\wass(X)$ starting at $\delta_x$. Then, the comparison angle $\widetilde{\angle}_{\mu_{s}\delta_x\sigma_{t}} $ at $\delta_x$ of
 the triangle $\Delta (\delta_x, \mu_s,\sigma_t)$ is a nondecreasing function of $s$ 
  and $t$. Consequently, the map $t \longrightarrow \dw(\mu_t, \sigma_t)/t$ is a 
  nondecreasing function as well.
 \end{lemm}
\begin{proof}
We set $d_m,d_s \leq + \infty$ the length of $(\mu_t)$ and $(\sigma_t)$ respectively and $\mu$, $\sigma$ the corresponding optimal dynamical plans. Thanks to the lifting lemma, 
we set $\Theta \in \Gamma
(\mu, \sigma)$ a dynamical plan such that, for given $s \leq d_m$ and $t \leq d_s$,
 $(e_s,e_t)_{\#}\Theta$ is an optimal plan. By definition of the Wasserstein distance, we get, for any $s' 
 \leq s$ and $t'\leq t$, the following 
 estimate
 $$ \dw^2(\mu_{s'},\sigma_{t'}) \leq \int d^2(\gamma(s'),\gamma'(t'))
\, \Theta(d\gamma, d\gamma').$$
 Now, the fact that $X$ is nonpositively curved yields
 \begin{eqnarray*} d^2(\gamma(s'),\gamma'(t')) & \leq  &
 \frac{s'^2}{s^2}d^2(x, \gamma(s)) + \frac{t'^2}{t^2}d^2(x,
 \gamma'(t)) \\ & & {}- 2 \frac{s't'}{st}d(x,\gamma(s))d(x,\gamma'(t))\cos 
 \widetilde{\angle}_{\gamma(s)x\gamma'(t)}.
 \end{eqnarray*}
 where $\widetilde{\angle}_{\gamma(s)x\gamma'(t)}$ is the comparison angle at $x$ (here, we use the 
 fact that the initial measure is a Dirac mass). 
 By integrating this inequality against $\Theta$, we get
 \begin{multline*} \dw^2(\mu_{s'},\sigma_{t'})  \leq 
 s'^2 + t'^2 \\-2 \frac{s't'}{st} \int d(x,\gamma(s))d(x,\gamma'(t))\cos 
 \widetilde{\angle}_{\gamma(s)x\gamma'(t)}
 \,\Theta(d\gamma,d\gamma').
 \end{multline*}
 We conclude by noticing that the inequality above is an equality when $s=s'$ and $t=t'$, so we 
 get 
  $$  \dw^2(\mu_{s'},\sigma_{t'}) \leq s'^2 + t'^2
 -2 s't'\cos \widetilde{\angle}_{\mu_{s}\delta_x\sigma_{t}}$$
 which is equivalent to the property $ \widetilde{\angle}_{\mu_{s'}\delta_x\sigma_{t'}} \leq  \widetilde{\angle}_{\mu_{s}\delta_x\sigma_{t}} $. The remaining statement follows readily.
  \end{proof}
  
 \begin{lemm}\label{lemm:suppi} Let $(\mu_t)$ be a unitary geodesic, possibly a ray, starting at 
$\delta_y$ and $\delta_x 
\neq \delta_y$. For any $\theta \in \wass(X)$ such that $\theta  \neq \delta_x,
\delta_y$, the comparison angle at $\theta$ satisfies
$$ \cos \widetilde{\angle}_{\delta_x\theta\delta_y} = \frac{1}{\dw(\delta_x,\theta) \dw
(\delta_y,\theta)} \int d(x,z)d(y,z) \cos \widetilde{\angle}_{xzy} \,\theta(dz).$$
Moreover, given two  nonnegative numbers $s\neq t$, the following inequality holds for $0<t<T$
$$ \widetilde{\angle}_{\mu_0\mu_t\delta_x} + \widetilde{\angle}_{\mu_T\mu_t\delta_x} \geq \pi.$$
\end{lemm}
\begin{proof}
For any $z \in X$, the following equality holds 
$$ d^2(x,y) = d^2 (x,z) + d^2(y,z) - 2 d(x,z)d(y,z) \cos \widetilde{\angle}_{xzy}.$$
By integrating this inequality against $\theta$, we get the first statement by definition of the comparison 
angle.
Let $\mu$ be the unique optimal dynamical coupling that induces $(\mu_t)$. The first step of the 
proof is to get the equality below:
\begin{multline}\label{equa:angleint}
\cos \widetilde{\angle}_{\mu_T\mu_t\delta_x}  =\\ \frac{1}{\dw(\mu_T, \mu_t) \dw(\mu_t,
\delta_x)} \int 
d(\gamma(t),\gamma(T))d(\gamma(t),x) \cos \widetilde{\angle}_{x 
\gamma(t)\gamma(T)} \, \mu (d \gamma)
\end{multline}
For $\gamma \in \supp \,\mu$, the following equality holds
\begin{multline*}
d^2(\gamma(T), x ) =
 d^2(\gamma(t),x) + d^2(\gamma(t),\gamma(T)) \\- 2 d(\gamma(t),x)d(\gamma(t),
\gamma(T))\cos \widetilde{\angle}_{x\gamma(t)\gamma(T)}.
\end{multline*}
By integrating this equality against $\mu$, we get (\ref{equa:angleint}). Now, using that $X$ is 
nonpositively curved, we have $ \widetilde{\angle}_{x\gamma(t)\gamma(T)}  + 
\widetilde{\angle}_{x\gamma(t)y} \geq \pi$. This gives 
\begin{multline*}
\cos \widetilde{\angle}_{\mu_T\mu_t\delta_x}   \\
        \begin{array}{rl}
        \leq   & \frac{-1}{\dw(\mu_T, \mu_t) \dw(\mu_t,\delta_x)} 
           \int d(\gamma(t),\gamma(T))d(\gamma(t),x) \cos \widetilde{\angle}_{x 
\gamma(t)y} \, \mu (d \gamma)   \\
  & \\
  \leq &\frac{-1}{(T-t)\dw(\mu_t,\delta_x)} 
   \int \frac{T-t}{t}d(\gamma(t),y)d(\gamma(t),x) \cos \widetilde{\angle}_{x 
\gamma(t)y} \, \mu (d \gamma) \\
 & \\
   \leq &\frac{-1}{\dw(\mu_t,\delta_y)\dw(\mu_t,\delta_x)} 
   \int d(\gamma(t),y)d(\gamma(t),x) \cos \widetilde{\angle}_{x 
\gamma(t)y} \, \mu (d \gamma) \\
 & \\
 \leq &- \cos \widetilde{\angle}_{\mu_0\mu_t\delta_x}
  \end{array}
  \end{multline*}
  where the last inequality follows from the first statement and the result is proved.
\end{proof}

As a consequence, we get the following result.

\begin{prop}\label{prop:convray}  Given $\ep>0$, $a>0$, and $R>0$, there exists a constant $T= 
T(\ep,a,R)>0$ such that the followings holds: for any $x,y \in X$ such that $d(x,y)=a$ and a 
unitary 
geodesic (possibly a ray)  $(\mu_t)$ of length greater than $T$ and starting at $\delta_y$, if 
$(\sigma^s_t)$ is the unitary geodesic from $\delta_x$ to $\mu_s$ then
$$ \dw(\sigma^s_R,\sigma^{s'}_R) <\ep$$
 for any $s'>s>T$.

 In particular, if $(\mu_t)$ is a ray and $s$ goes to infinity, $(\sigma^s)_{s \geq 0}$ converges uniformly on compact subsets to the 
 unitary ray $\mu_{\delta_x,
 \xi}$ where $\xi$ is the asymptote class of $(\mu_t)$ 
 (see Figure \ref{fig:bound}).
\end{prop}

\begin{figure}[htp]\begin{center}
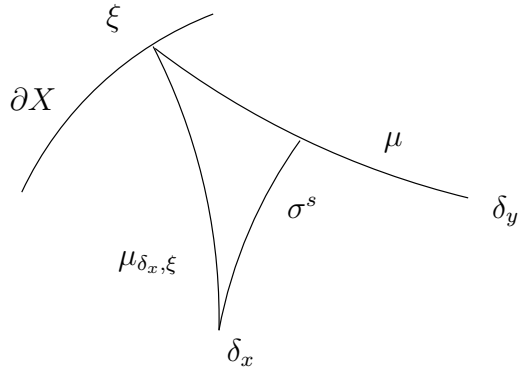
\caption{Uniform convergence of $\sigma^s$ on compact subsets.}
\label{fig:bound}
\end{center}\end{figure}
\begin{proof}
Under these assumptions, the comparison angle $\widetilde{\angle}_{\delta_y\mu_s\delta_x}$ is 
arbitrary 
small 
provided $s$ is sufficiently large. Consequently, thanks to Lemma \ref{lemm:suppi}, 
$\widetilde{\angle}_{\mu_{s'}\mu_s\delta_x}$ is close to $\pi$;  therefore the comparison angle 
$\widetilde{\angle}_{\mu_s\delta_x \mu_{s'}}$ is small. This gives the first part of the result since 
$$ \widetilde{\angle}_{\sigma^s_R\delta_x \sigma^{s'}_R} \leq 
\widetilde{\angle}_{\mu_s\delta_x\mu_{s'}}$$
thanks to Lemma \ref{lemm:triangle}. Using Lemma \ref{lemm:triangle} again, it only remains to prove the pointwise convergence of $(\sigma^s)_{s \geq 0}$ to 
$\mu_{\delta_x,\xi}$. Thanks to the asymptotic formula, there exists $C>0$ such that 
\begin{equation}\label{equa:bound}
\dw
((\mu_{\delta_x,\xi})_t, \mu_t) \leq C
\end{equation}
 for any nonnegative number $t$. Finally, we conclude by using $s'$ 
sufficiently large and the bound 
$$
 \dw( (\mu_{\delta_x,\xi})_R, \sigma^{s}_R) \leq  
  \dw( (\mu_{\delta_x,\xi})_R, 
\sigma^{s'}_R) +  
\dw( \sigma^{s}_R, \sigma^{s'}_R)
$$
where the same reasoning as above and (\ref{equa:bound}) show that the first term on the right-hand side is small provided $s'$ is large.
\end{proof}
Now, we can prove that the topology above is well-defined and does not depend on 
the choice of the base point $\delta_x$. This the content of the lemma below.

\begin{lemm}Given two positive numbers $R, \ep$ and $y \in X$, $ \xi \in U(x, \eta,R , \ep)\cap 
\partial \wass(X)$, there exists $S, \ep' >0$ such that
$$ U(y,\xi, S, \ep') \subset U(x, \eta,R , \ep).$$
\end{lemm}
\begin{proof}
We set $\alpha = \ep - \dw((\mu_{\delta_x,\eta})_R, (\mu_{\delta_x,\xi})_R)>0$. Let $\theta 
\in 
U(y, \xi,S,\ep')$ and $\Theta$ (respectively $\Xi$) be the unitary geodesic $\mu_{\delta_y,
\theta} $ (respectively the unitary ray $\mu_{\delta_y,\xi} $). We have

\begin{multline*}
\dw ((\mu_{\delta_x,\theta})_R, (\mu_{\delta_x,\eta})_R) \leq 
\shoveright{\dw ((\mu_{\delta_x,\theta})_R, (\mu_{\delta_x,\Theta_S})_R) +
\dw((\mu_{\delta_x,\Theta_S})_R, (\mu_{\delta_x,\Xi_S})_R)}\\
 {}+ \dw ((\mu_{\delta_x,\Xi_S})_R, (\mu_{\delta_x,\xi})_R) + \dw ((\mu_{\delta_x,
 \xi})_R, (\mu_{\delta_x,\eta})_R) 
\end{multline*}
The first and the third term on the right-hand side are smaller than $\alpha/3$ for large $S$ 
thanks 
to Proposition \ref{prop:convray} while the second term is smaller  than $\alpha/3$ for large $S$ 
and small $\ep'$ thanks to lemma \ref{lemm:triangle}.
\end{proof}

\section{The boundary of $\wass(X)$ viewed as a set of measures}
To state the main result of this section, we first need to introduce a definition.

\begin{defi} We set
$$\begin{array}{rccl}
  Am : &  \mathscr{R}_{\delta_x,1}(\wass(X))& \longrightarrow &\pr_1(c\partial X) \\
                            &    (\mu_t)         & \longmapsto       & \mu_{\infty}
  \end{array}
  $$
  the map that sends a unitary ray starting at $\delta_x$ to its asymptotic measure.
  \end{defi}
  
  The main result of this part is the following theorem.
\begin{theo}\label{theo:homeobound} The map $Am : \mathscr{R}_{\delta_x,1}(\wass(X)) 
\longrightarrow 
\pr_1(c\partial X)$ induces a homeomorphism from 
$\partial\wass(X) $ onto $\pr_1(c\partial X)$.
\end{theo}

Note that a straightforward consequence of the result above is
\begin{coro}\label{coro:homeo} Let $d$ be a metric on $\partial X$ that induces the cone topology on $\partial X$ 
and 
$d_C$ the cone metric induced by 
$d$ on $c\partial X$ (see (\ref{equa:conemetric}) for a definition). 
Let us denote by $\wass(c \partial X)$ the quadratic Wasserstein space over the Polish 
space $(c\partial 
X, d_C)$. Then, $\partial\wass(X) $ is homeomorphic to 
the subset of probability measures with unitary speed in $\wass(c \partial X)$.
\end{coro}

\begin{rema}
In particular we get the more symmetric result that $c\partial\wass(X)$ 
is homeomorphic to $\wass(c \partial X)$.
\end{rema}

The rest of this part is devoted to the proof of the theorem above. Recall that we have proved in Lemma \ref{lemm:existray} that both ${Am}$ and the map 
$\widetilde{Am}: \partial \wass(X) \rightarrow 
\pr_1(c\partial X)$ it induces are bijective.

The proof of Theorem \ref{theo:homeobound} is in two steps. First, we prove that the map $Am$ is 
a homeomorphism. Then, we use this fact to prove that $\widetilde{Am}$ is a homeomorphism as well.

We start the proof with a definition.

\begin{defi} Let $x \in X$. We denote by
 \begin{eqnarray*}
 ODT_{x} & = &\big\{ \mu \in \pr(\mathscr{R}_x(X)) ; \int s^2(\gamma) \mu(d\gamma)=1\big\}
 \end{eqnarray*}
  the set of unitary dynamical transport plans endowed with the weak topology. We also set
  $$\begin{array}{rccl}
  {(e_t)_{\geq 0}}_{\#} : & ODT_{x} & \longrightarrow & \mathscr{R}_{\delta_x,1}(\wass(X)) \\
                            &    \mu          & \longmapsto       & (\mu_t)  
  \end{array}
  $$
and
  $$\begin{array}{rccc} 
{e_{\infty}}_{\#} : & ODT_x &\longrightarrow & \pr_1(c\partial X) \\
                  &         \mu   & \longmapsto    &  {{\phi_x}}_{\#} \mu
        \end{array}$$
 where $\phi_x$ is defined in (\ref{equa:phix}).
\end{defi}
Thanks to Lemma \ref{lemm:existray}, we have the following commutative diagram where all the maps are one-to-one.
   $$ \begin{diagram}  
ODT_x & \rOnto^{\;\;\;\;{e_{\infty}}_{\#}\;\;\;\;\;\;\;\;\;} &  \pr_1(c\partial X) \\
     \dOnto^{({e_t)_{t \geq 0}}_{\#} \;\;} &  \ruOnto^{Am\;\;}& \\  
     \mathscr{R}_{\delta_x,1} (\wass(X)) & &
     \end{diagram}
     $$
     
 We first prove that
 \begin{lemm} The map ${e_{\infty}}_{\#} $   is a homeomorphism onto $\pr_1(c\partial X)$.
 \end{lemm} 
\begin{proof}
The map $\phi_x : \mathscr{R}_{x}(X) \longrightarrow c \partial X$ is a homeomorphism. Therefore 
it induces a homeomorphism between $ODT_x$ and $\pr_1(c\partial X)$  when endowed with the 
weak topology.
\end{proof}
 
 \begin{lemm}\label{lemm:ODT} The map $ {(e_t)_{\geq 0}}_{\#}$ is a continuous map.
 \end{lemm}
  \begin{proof} Since the spaces we consider are metrizable, we just have to prove the sequential 
  continuity. Consequently, we are given a sequence $(\mu_n)_{n \in \mathbb{N}}$ such that $\mu_n \rightharpoonup \mu$ in 
  $ODT_x$. Now, since $e_t:  \mathscr{R}_{x}(X) \longrightarrow X$ is a continuous map, we get 
  that ${e_t}_{\#}\mu_n \rightharpoonup {e_t}_{\#}\mu$ in $\mathscr{P}(X)$. By definition of $ODT_x$, we have
  $$ \int s^2(\gamma) \,\mu_n(d\gamma) = \int  s^2(\gamma) \,\mu(d\gamma) =1.$$
  Since $ \int s^2(\gamma) \,\mu_n(d\gamma) = \int d^2(x,\gamma(1))\,
  \mu_n(d\gamma)$, the equality above implies the convergence of the second moment. Namely, we 
  have
  $$ \int d^2(x,\gamma(t)) \, \mu_n(d\gamma) = t^2 \int d^2(x,\gamma(1))\,\mu_n(d\gamma)= \int 
  d^2(x,\gamma(t)) \, \mu(d\gamma).$$
  This implies the convergence of ${e_t}_{\#}\mu_n$ to  ${e_t}_{\#}\mu$ with respect to the 
  Was\-serstein distance (see for instance \cite{Villani}, Theorem 6.9). Thus, we have proved the pointwise 
  convergence of rays. Now, since $t \longrightarrow \dw ({e_t}_{\#}\mu_n,{e_t}_{\#}\mu)$ is 
  nondecreasing as proved in Lemma \ref{lemm:triangle}, we get the result.
\end{proof} 
 
 We end the first part of the proof with the following lemma.
 \begin{lemm}\label{lemm:ODT2}  The map $ {(e_t)_{\geq 0}}_{\#}$ is a homeomorphism.
  \end{lemm}
  \begin{proof} Since the topology of both  $\mathscr{R}_{\delta_x,1}(\wass(X))$ and $ODT_x$  is induced 
  by a metric, it is 
  sufficient to prove that ${(e_t)_{\geq 0}}_{\#} $ is a proper map. Moreover, we just have to prove 
  sequential compactness.  We set $K $ a compact subset of $  \mathscr{R}_{\delta_x,1}
  (\wass(X))$. Let $(\mu_n)_{n \in \mathbb{N}} \in {(e_t)_{\geq 0}}_{\#}^{-1}(K)$
. We first notice that $(\mu_n)_{n \in \mathbb{N}}$ is tight. Indeed, by assumption on $K$, the sequence 
  $({e_1}_{\#}\mu_n)_{n \in \mathbb{N}}$ is tight in $\pr(X)$. Therefore, by arguing as in the end of the proof of 
  Proposition \ref{prop:displacement}, we obtain the claim. Consequently, since $\mathscr{R}_x(X)$ 
  is a Polish space, we can apply Prokhorov's theorem to get a converging subsequence 
  $(\mu_{n_k})_{k \in \mathbb{N}}$ to $\widetilde{\mu}$. It remains to prove that $\widetilde{\mu} \in ODT_x$, namely that 
  $\int s^2(\gamma) \, \widetilde{\mu}(d\gamma)=1$. Since $K$ is compact, we can also assume
  without loss of generality that $({e_t}_{\#}\mu_{n_k}) \longrightarrow (\bar{\mu}_t)$ in 
  $\mathscr{R}_{\delta_x,1} (\wass(X))$. Moreover, since ${e_t}_{\#}\mu_{n_k} \rightharpoonup {e_t}_{\#}
  \widetilde{\mu}$ for any $t$, we get  ${(e_t)_{\geq 0}}_{\#}
  (\widetilde{\mu})=(\bar{\mu}_t)$. Therefore,  ${e_1}_{\#}\mu_{n_k} \longrightarrow 
  {e_1}_{\#}\widetilde{\mu}$ in $\wass(X)$. This implies the convergence of the second moment $\int 
  d^2(x,\gamma(1)) \, \mu_{n_k}(d\gamma)=\int s^2(\gamma)\, \mu_{n_k}(d\gamma)=1$ (see for 
  instance \cite{Villani}, Theorem 6.9) and the result is proved.
\end{proof}  
  
 We are now in position to prove Theorem \ref{theo:homeobound}. We set $p_{\partial \wass}$ the 
 canonical projection on $ \partial\wass(X) $. We have the following commutative diagram.
   $$ \begin{diagram}  
 \mathscr{R}_{\delta_x, 1}(\wass(X)) & \rOnto^{\;\;\;\;Am\;\;\;\;\;\;}  &  \pr_1(c\partial X) \\
     \dOnto^{p_{\partial \wass}\;\;\;} &  \ruOnto^{\widetilde{Am}\;\;}& \\  
     \partial\wass(X) & &
     \end{diagram}
     $$
 We have seen at the beginning of the proof that all the maps above are one-to-one. To conclude, 
 it remains to prove that $p_{\partial \wass}^{-1}$ is a continuous map. Since $\partial\wass(X)$ 
 is first-countable (see Remark \ref{rema:firstc}), it is sufficient to prove sequential continuity. To 
 this aim, let $\xi_n \longrightarrow \xi$ in $\partial\wass(X)$ and $({e_t}_{\#}{\mu_n}), (\mu_t) \in  
 \mathscr{R}_{\delta_x, 1}(\wass(X)) $ such that  $p_{\partial \wass}(({e_t}_{\#}{\mu_n}))=\xi_n$ and 
 $p_{\partial \wass}((\mu_t))=\xi$. Recall that under these assumptions, the map $t \longrightarrow 
 \dw({e_t}_{\#}{\mu_n},\mu_t)$ is nondecreasing (see Lemma \ref{lemm:triangle}), thus we 
 just have to show the pointwise convergence of $({e_t}_{\#}{\mu_n})$. This pointwise convergence 
 follows readily from the definition of the cone topology on $\overline{\wass(X)}$.

\bibliographystyle{smfalpha}
\bibliography{biblio.bib}

\end{document}